\documentclass{amsart}
\usepackage{amsmath,amsfonts,amsthm,amssymb,amscd,stmaryrd,url,amstext, amsxtra,amsopn}
\usepackage{mathtools}
\usepackage{xcolor}

\usepackage{hyperref}

\DeclareFontFamily{OT1}{rsfs}{}
\DeclareFontShape{OT1}{rsfs}{n}{it}{<-> rsfs10}{}
\DeclareMathAlphabet{\mathscr}{OT1}{rsfs}{n}{it}

 \DeclareMathOperator{\tr}{tr}
 
 \DeclareMathOperator{\Li}{Li}
 \DeclareMathOperator{\Gal}{Gal}
 \DeclareMathOperator{\GL}{GL}

 \newcommand{\mymod}[1]{(\operatorname{mod} #1)}

\def\q{\mathfrak{q}}
\def\p{\mathfrak{p}}
\def\a{\mathfrak{a}}

\def\m{\mathfrak{m}}

\def\c{\mathfrak{c}}

\def\norm{{\rm N}}       

\makeatletter
\let\@@pmod\pmod
\DeclareRobustCommand{\pmod}{\@ifstar\@pmods\@@pmod}
\def\@pmods#1{\mkern4mu({\operator@font mod}\mkern 6mu#1)}
\makeatother

\newtheorem{theorem}{Theorem}
 \newtheorem{corollary}{Corollary}[theorem]

   \theoremstyle{remark}
 \newtheorem*{remark}{Remark}
 
\reversemarginpar 

\begin{document}

\title[Square-free orders for CM elliptic curves modulo $p$]{Square-free orders for CM elliptic curves modulo $p$ in short intervals}

\dedicatory{In memory of my friend, Richard}

\author{Peng-Jie Wong}
\address{Department of Applied Mathematics\\
National Sun Yat-Sen University\\
Kaohsiung City, Taiwan}
\email{pjwong@math.nsysu.edu.tw}
\subjclass[2010]{11N45, 11G05, 11N36} 
\thanks{Statements and Declarations: The author was an NCTS postdoctoral fellow, and he was also supported by a PIMS postdoctoral fellowship and the University of Lethbridge during conducting part of this research. He is currently supported by the NSTC grant 114-2628-M-110-006-MY4.}

\keywords{CM elliptic curves, square-free orders, the cyclicity problem}
\date{}

\begin{abstract}
Let $E$ be a CM elliptic curve over $\Bbb{Q}$. We refine the work of Cojocaru on the asymptotic formulae for the number of primes $p\le x$ for which the reduction modulo $p$ of $E$ is of square-free order. Also, we derive an unconditional short interval variant for the asymptotics. Compared to the estimate derived from the generalised Riemann hypothesis, the presented result is valid for even shorter intervals. Furthermore, we improve the short interval variant of the cyclicity problem for CM elliptic curves previously obtained by the author.
\end{abstract}

\maketitle

\section{Introduction}

Let $E$ be an elliptic curve defined over $\Bbb{Q}$, and let $N_E$ be its conductor. For a prime $p$ of good reduction, we let $\bar{E}$ be the reduction of $E$ modulo $p$ and $\bar{E}(\Bbb{F}_p)$ be the group of rational points of $\bar{E}$ over $\Bbb{F}_p$. By Hasse's bound, there is an integer $a_p$ such that $|a_p|\le 2 \sqrt{p}$ and $|\bar{E}(\Bbb{F}_p)|=p+1-a_p$. For $p\ge 5$, $p$ is called supersingular if $a_p=0$; otherwise, $p$ is ordinary.

Since Lang and  Trotter \cite{LT} formulated an elliptic curve analogue of Artin's primitive root conjecture, the study of the structure of $\bar{E}(\Bbb{F}_p)$, as $p$ varies, has attracted many mathematicians. For instance,  Cojocaru \cite{Co08} determined an asymptotic formula for
\begin{equation}\label{Co-01}
h_E(x,\Bbb{Q})=\#\{p \le x \mid \text{$p\nmid N_E$, $a_p\neq 0$, and $|\bar{E}(\Bbb{F}_p)|$ is square-free}\}
\end{equation}
when $E$ has complex multiplication (CM) by the full ring of integers $\mathcal{O}_K$ of an imaginary quadratic field $K$.  More precisely, she showed that 
\begin{equation}\label{Co-02}
h_E(x,\Bbb{Q})=\delta_E \Li(x) +  O\Big( \frac{x}{(\log x)(\log\log ((\log x)/N_E^2))}\frac{\log\log x}{\log ((\log x)/N_E^2)}\Big),
\end{equation}
where $\Li(x) =\int_2^x \frac{dt}{\log t}$ is the logarithmic integral function,
\begin{equation}\label{delta}
\delta_E=\frac{1}{2} \sum_{\a,m}\frac{\mu(\a)\mu(m)}{[K(E[\a^2])K(E[m]):K]},
\end{equation}
the sum is over square-free integral ideals $\a$  of $\mathcal{O}_K$ composed of degree-one unramified prime ideals and over square-free $m\in \Bbb{N}$, and $\mu(\a)$ is the number field analogue of the M\"obius function for $K$. Here, as later, $E[m]$ (resp., $E[\a]$) denotes the group of $m$-torsion points (resp., the group of $\a$-torsion points) of $E$.
Moreover, under the generalised Riemann hypothesis (GRH) for the Dedekind zeta functions of the division fields of $E$, Cojocaru established
\begin{equation}\label{Co-02-GRH}
h_E(x,\Bbb{Q})=\delta_E \Li(x) +   O( x^{5/6}(\log x)^2(\log (N_E x))^{1/3}).
\end{equation}
Recall that as $E$ is defined over $\Bbb{Q}$, the field $K$ is one of the nine imaginary quadratic fields of class number 1. It worth remarking that by \cite[Theorem 1.3]{Co08}, one knows $\delta_E>0$ whenever $K$ is $\Bbb{Q}(\sqrt{-11})$,  $\Bbb{Q}(\sqrt{-19})$,  $\Bbb{Q}(\sqrt{-43})$,  $\Bbb{Q}(\sqrt{-67})$, or  $\Bbb{Q}(\sqrt{-163})$. 

One of the objectives of this article is to prove the following improved estimates for $h_E(x,\Bbb{Q})$.

\begin{theorem}\label{main-thm1}
Let $E/\Bbb{Q}$ be an  elliptic curve  of conductor $N_E$ and with complex multiplication by the  ring of integers $\mathcal{O}_K$ of an imaginary quadratic field $K$.  Then  for any  $A>0$,  we have
\begin{align}\label{main-1}
h_E(x,\Bbb{Q})=\delta_E \Li(x)+O_A\Big(N_E\frac{x}{(\log x)^A}\Big),
\end{align}
where $\delta_E$ is defined as in \eqref{delta}. 

Moreover, assuming  GRH, for any $\eta\in(0, \frac{1}{3})$, we have 
\begin{equation}\label{improved-hQ}
h_E(x,\Bbb{Q}) =\delta_E \Li(x) +
O_\eta \Big( x^{5/6}  \frac{(\log(N_E x ) )^{1/3} }{ (\log x)^{2/3}}\Big)
\end{equation}
 uniformly in $N_E\le x^\eta$. 
\end{theorem}

\begin{remark}
For general elliptic curves $E/\Bbb{Q}$, one may instead consider
$$
\pi_E^{\mathrm{SF}} (x) = \#\{ p\le x \mid  \text{$p\nmid N_E$  and $|\bar{E}(\Bbb{F}_p)|$ is square-free} \}.
$$
For $E$ with CM, the methods used in \cite{Co08} and this article can be adapted to study $\pi_E^{\mathrm{SF}} (x)$. When $E$ is non-CM, Cojocaru \cite{Co02} showed how to derive an asymptotic formula for $\pi_E^{\mathrm{SF}} (x)$  under the generalized Riemann hypothesis, Artin's holomorphy conjecture, and the pair correlation conjecture. However, the precise asymptotic of $\pi_E^{\mathrm{SF}} (x)$ is still not known unconditionally for this case. Nonetheless, in \cite{Ge}, Gekeler proved that the  conjectural asymptotic for $\pi_E^{\mathrm{SF}} (x)$ holds on average over elliptic curves $E$. Moreover, in \cite{ADHT}, Akhtari, David, Hahn, and Thompson unconditionally established the predicted upper bound for $\pi_E^{\mathrm{SF}} (x)$ and showed that the average results are compatible with the conjectural asymptotics at the level of the constants. 
\end{remark}

It is worth mentioning that in \cite{Co08}, after establishing \eqref{Co-02} and \eqref{Co-02-GRH}, Cojocaru studied an elliptic curve analogue of Linnik's problem as follows. Denoting $p_E$ the smallest ordinary prime $p$ for which $|\bar{E}(\Bbb{F}_p)|$ is square-free, she showed that if $\delta_E\neq 0$, then $p_E=O(\exp(e N_E^{3}))$, unconditionally, and $p_E=O_{\epsilon}((\log N_E)^{2+\epsilon})$  under GRH. In light of her work, we prove the following refined unconditional bound for  $p_E$.

\begin{corollary}
Let $\epsilon\in(0,1)$. If $\delta_E\neq 0$, then
$$
p_E=O_{\epsilon}(\exp(N_E^{\epsilon})).
$$
\end{corollary}

\begin{proof}
For $\epsilon\in(0,1)$, we apply Theorem \ref{main-thm1} with $A= 2/\epsilon$ to obtain
$$
h_E(x,\Bbb{Q})=\delta_E \Li(x) +O_{\epsilon}\Big(N_E \frac{x}{(\log x)^{2/\epsilon}}\Big).
$$
As $\delta_E\gg \frac{1}{\log\log N_E}$ (see \cite[Sec. 7]{Co08}), if we choose $x= \exp(N_E^{\epsilon})$, then $h_E(x,\Bbb{Q})>0$  for $N_E$ sufficiently large with respect to $\epsilon$.  Thus, there is an ordinary prime $p_E=p = O_{\epsilon}(\exp(N_E^{\epsilon}))$  such that $|\bar{E}(\Bbb{F}_p)|$ is square-free.
\end{proof}

Furthermore, as may be noticed, the estimates \eqref{Co-02},  \eqref{Co-02-GRH}, \eqref{main-1}, and \eqref{improved-hQ}  present an  ``elliptic'' analogue of the prime number theorem. Indeed,  a strong form of the prime number theorem states that for any $A>0$, one has
\begin{align}\label{PNT}
\pi(x)  =\Li(x)+O\Big(\frac{x}{(\log x)^A}\Big);
\end{align}
under the Riemann hypothesis, one has
\begin{align}\label{RH}
\pi(x)  =\Li(x)+O(x^{1/2}\log x).
\end{align}
In a slightly different vein, as \eqref{PNT} gives
$$
\pi(2x)-\pi(x)  =\Li(2x)- \Li(x)+O\Big(\frac{x}{(\log x)^A}\Big),
$$
one may ask for an asymptotic formula of the distribution of primes in short intervals  when  $(x,2x]$  is replaced by $(x,x+h]$ for some $x^{1-\delta}\le h  \le x$ with $\delta\in [0,1)$. 
Under the Riemann hypothesis, \eqref{RH} implies that for any $A>0$,
\begin{align}\label{RH-s}
\pi(x+h)-\pi(x)  =\Li(x+h)- \Li(x)+O\Big(\frac{h}{(\log x)^A}\Big)
\end{align}
whenever $x^{1-\delta}\le h  \le x$ with $\delta\in [0,\frac{1}{2})$.    Although the Riemann hypothesis is still out of reach, there are some progresses towards \eqref{RH-s}. For example, Huxley \cite{Hu71} showed that \eqref{RH-s} is valid for $x^{1-\delta}\le h  \le x$ with $\delta\in [0,\frac{5}{12})$.
From the above discussion, one may further consider
\begin{equation}\label{def-h-SI}
h_E(x,h,\Bbb{Q})=\#\{x< p \le x +h \mid \text{$p\nmid N_E$, $a_p\neq 0$, and $|\bar{E}(\Bbb{F}_p)|$ is square-free}\}.
\end{equation}
By the virtue of \eqref{main-1}, for any $A>0$, we have
\begin{equation*}
h_E(x,x,\Bbb{Q})  =\delta_E (\Li(2x)- \Li(x))+ O_A\Big( N_E\frac{x}{(\log x)^A}\Big).
\end{equation*}
Furthermore, by \eqref{Co-02-GRH}, under GRH, for any $x^{1-\delta}\le h  \le x$ with $\delta\in[0,\frac{1}{6})$, one has
\begin{equation}\label{square-free-GRH}
h_E(x,h,\Bbb{Q})  =\delta_E (\Li(x+h)- \Li(x))+ O\Big((\log N_E)^{1/3} \frac{h}{(\log x)^A}\Big)
\end{equation}
for any $A>0$. However, it is apparent that we are far from reaching GRH, which leads one to ask whether an unconditional version of \eqref{square-free-GRH} can be obtained. In this article,  we shall prove the following estimate for $h_E(x,h,\Bbb{Q})$, which is not only unconditional but also valid for shorter intervals compared to   \eqref{square-free-GRH}.

\begin{theorem}\label{main-thm2}
Let $E/\Bbb{Q}$ be an  elliptic curve  of conductor $N_E$ and with complex multiplication by the  ring of integers $\mathcal{O}_K$ of an imaginary quadratic field $K$. Let $A> 0$ and  $0\le \delta<\frac{1}{5}$. Then 
for any  $x^{1-\delta}\le h\le x$,  we have
\begin{align*}
h_E(x,h,\Bbb{Q})
=\delta_E (\Li(x+h)-\Li(x))+O_A\Big( N_E\frac{h}{(\log x)^A}\Big),
\end{align*}
where $\delta_E$ is defined as in \eqref{delta}. 
\end{theorem}

\begin{remark}
The proofs of \eqref{Co-02} and \eqref{Co-02-GRH} in \cite{Co08} respectively rely on the effective versions of the Chebotarev density theorem established by Lagarias-Odlyzko  \cite{LO77}  and  M.R. Murty-V.K. Murty-Saradha \cite{MMS88}. In contrast, our proofs of \eqref{main-1} and Theorem  \ref{main-thm2} make a use of the Bombieri-Vinogradov theorem for number fields due to Huxley \cite{Hu86} (see also Theorem \ref{Huxley} below)  and its short interval variant (Theorem \ref{BVSI-Li}). This is inspired by the works of  Akbary-V.K. Murty \cite {AM} and M.R. Murty \cite{MR83} on the cyclicity problem discussed below. The key observation is that the Chebotarev conditions involved in the proofs can be translated into congruence conditions over certain ray class groups via Artin reciprocity. Nonetheless, to establish \eqref{improved-hQ}, we still require the effective result of  M.R. Murty, V.K. Murty, and Saradha \cite{MMS88}. The new input in our argument is to consider a finer splitting \eqref{hE-expression-finer} below and invoke  the  Brun-Titchmarsh inequality (for $K$) due to Hinz and Lodemann \cite{HL93} to control the ``middle range'' $\mathcal{N}_0(x, y , x^{\theta})$.
\end{remark}

Last but not least, we recall that rooted in the Lang-Trotter conjecture, there is a problem of finding  an asymptotic formula for the number of primes $p\le x$ for which the reduction modulo $p$ of $E$ is cyclic.\footnote{Note that if  $|\bar{E}(\Bbb{F}_p)|$ is square-free, then $\bar{E}(\Bbb{F}_p)$ is cyclic. So, the study of the square-freeness of $|\bar{E}(\Bbb{F}_p)|$ may  be seen as an intermediate problem between such a cyclicity problem and Koblitz's conjecture \cite{Ko} on the primality of $|\bar{E}(\Bbb{F}_p)|$.} More precisely, one may consider
$$
\pi_{c}(x,E)=\#\{p \le x \mid \text{$p\nmid N_E$ and $\bar{E}(\Bbb{F}_p)$ is cyclic}\}. 
$$
This has been studied by Akbary, Cojocaru, Gupta, M.R. Murty, V.K. Murty, and Serre (see  \cite{Wong20} for a more detailed discussion and references therein). Similar to the above-discussed problem  concerning square-free orders for CM elliptic curves in short intervals, one can also consider $$
\pi_{c}(x+h,E)-\pi_{c}(x,E)=\#\{ x< p \le x+h \mid \text{$p\nmid N_E$ and $\bar{E}(\Bbb{F}_p)$ is cyclic}\}. 
$$
By the work of  Akbary and  V.K. Murty \cite{AM}, one has
$$
\pi_{c}(2x,E)-\pi_{c}(x,E) =\mathfrak{c}_E(\Li(2x)-\Li(x) )+O_{A,D}\Big(\frac{x}{(\log x)^A}\Big)
$$
uniformly in $N_E \le (\log x)^D$, where  $$
\mathfrak{c}_E=\sum_{m=1}^{\infty}\frac{\mu(m)}{[\Bbb{Q}(E[m]):\Bbb{Q}]},
$$
and the implied constant depends on $A$ and $D$. Moreover, under GRH, the work of Cojocaru and M.R. Murty \cite{CM04} yields 
\begin{equation}\label{CM}
\pi_{c}(x+h,E)-\pi_{c}(x,E) =\mathfrak{c}_E(\Li(x+h)-\Li(x) )+O_E\Big(\frac{h}{(\log x)^A}\Big)
\end{equation}
for any  $x^{1-\delta}\le h\le x$ with $0\le \delta<\frac{1}{4}$. Unconditionally, in \cite{Wong20}, the  present author showed that the estimate \eqref{CM} is valid whenever  $x^{1-\delta}\le h\le x$, with $0\le \delta<\frac{1}{25}$. In this article, we shall further sharpen such an unconditional result as follows.

\begin{theorem}\label{main-thm-cyclicity}
Let $E/\Bbb{Q}$ be an  elliptic curve of conductor $N_E$ and with complex multiplication by the  ring of integers $\mathcal{O}_K$ of an imaginary quadratic field $K$. Let $A> 0$ and  $0\le \delta<\frac{1}{5}$. Then 
for any  $x^{1-\delta}\le h\le x$,  we have
\begin{align*}
\pi_{c}(x+h,E)-\pi_{c}(x,E)
=\mathfrak{c}_E (\Li(x+h)-\Li(x))+O\Big(N_E(\log N_E)\frac{h}{(\log x)^A}\Big),
\end{align*}
where  the implied constant  depends on $\Bbb{Q}(E[2])$ and $A$.
\end{theorem}

\begin{remark} 
The key idea for the improvement is to use a number field analogue of the Brun-Titchmarsh inequality, the estimate \eqref{NTA-BT}. A similar estimate of such an analogue also plays a crucial role in proving Theorem \ref{main-thm2}, especially bounding  the first inner sum in \eqref{M-h-mid}. By these estimates, we then \emph{only} require the Bombieri-Vinogradov theorem for number fields and its short interval variant to control small suitable initial ranges through the sieving procedures.
\end{remark}

The rest of this article is arranged as follows. In the next section, we will discuss the strategy of proving Theorems \ref{main-thm1} and \ref{main-thm2}. We will prove Theorems \ref{main-thm1}, \ref{main-thm2}, and \ref{main-thm-cyclicity} in Sections \ref{proof-main-thm1}, \ref{proof-main-thm2}, \ref{proof-main-thm3}, respectively. We also note that throughout our proofs of \eqref{main-1} and Theorems \ref{main-thm2} and \ref{main-thm-cyclicity}, we will implicitly assume $N_E\le \exp(\sqrt{\log x})$ as the claimed estimates hold trivially otherwise.

\section{Overview of the proofs of Theorems \ref{main-thm1} and \ref{main-thm2}}

In this section, we shall collect some necessary facts and review the strategy of Cojocaru \cite{Co08}. Let $E/\Bbb{Q}$ be an elliptic curve with complex multiplication by the full ring of integers $\mathcal{O}_K$ of an imaginary quadratic field $K$. Let $N_E$, $E[m]$, and $E[\a]$ be defined as in the introduction. We recall that for any $m\in\Bbb{N}$, there is an injective Galois representation
$$
\rho_m :\Gal(K(E[m])/K )\rightarrow \GL_2(\Bbb{Z}/m\Bbb{Z})
$$ 
such that
\begin{equation}\label{tr}
\tr \rho_m(\sigma_{\p})\equiv a_{\p} \mymod{ m}
\end{equation}
and
\begin{equation}\label{det}
\det \rho_m(\sigma_{\p})\equiv \norm(\p) \mymod{ m},
\end{equation}
for any prime $\p\nmid m N_E$ of $K$, where $\sigma_{\p}$ denotes the Artin symbol at $\p$, $\norm(\p)$ is the norm of $\p$, and $a_{\p}= \norm(\p)+1 -|\bar{E}(\Bbb{F}_{\p})|$.  Following Cojocaru \cite{Co08}, we set
$$
D_{m}=\{g\in \Gal(K(E[m])/K )\mid  \det\rho_m (g) + 1 - \tr \rho_m(g) \equiv 0 \mymod{ m} \}.
$$
By a criterion of Deuring \cite{De41}, for any prime $p\ge 5$ that is of good reduction for $E$, $p$ is ordinary (i.e. $a_p\neq 0$) if and only if  $p$ splits in $K$. Moreover, if $p$ splits in $K$ and $\p$ is a prime of $K$  above $p$, then we have $\norm(\p)=p$ and $a_\p=a_p$. Thus,  from  \eqref{tr} and \eqref{det}, it follows that 
the condition $|\bar{E}(\Bbb{F}_p)|\equiv 0 \mymod{ m^2}$  is equivalent to $\sigma_\p \subseteq D_{m^2}.$ Hence, we obtain
\begin{align*}
&\#\{ p \le x \mid p\nmid m N_E, a_p\neq 0, |\bar{E}(\Bbb{F}_p)|\equiv 0\mymod{ m^2}\}\\
& = \frac{1}{2} \#\{\p\subset\mathcal{O}_K \mid  \norm(\p)\le x, \p\nmid mN_E, \sigma_{\p}\subseteq  D_{m^2} \}+ O\Big( \frac{\sqrt{x}}{\log x} +\log( mN_E) \Big),
\end{align*}
where the error comes from the degree-two primes $\p$ of $K$ such that $\norm(\p)\le x$ (there are at most $O(\sqrt{x}/\log x)$ such primes) and  ramified primes (cf. \cite[Eq. (3.2)]{AM}).

For the sake of convenience, we shall set
\begin{equation}\label{def-pi-D}
\pi_{D_{m^2}}(x)=\#\{\p\subset \mathcal{O}_K \mid \norm(\p)\le x, \p\nmid mN_E, \sigma_{\p}\subseteq  D_{m^2} \}.
\end{equation} 
Similar to \cite[Sec. 2]{Co08}, by the inclusion-exclusion principle, Hasse's bound, and the above discussion,  we can express $h_E(x,\Bbb{Q})$ as
\begin{align}\label{hE-expression}
 \begin{split}
h_E(x,\Bbb{Q})
&=\sum_{ m \le 2\sqrt{x}} \mu(m) \#\{ p\le x \mid   p\nmid mN_E, a_p\neq 0,  m^2\mid |\bar{E}(\Bbb{F}_p)| \}\\
& =\mathcal{N}(x,y)+\mathcal{M}(x,y,2\sqrt{x})+ O\Big( \frac{\sqrt{x} y}{\log x} +y\log( yN_E) \Big),
 \end{split}
\end{align} 
where $h_E(x,\Bbb{Q})$ is as introduced in \eqref{Co-01},
$$
\mathcal{N}(x,y) =\frac{1}{2}\sum_{m\le y}\mu(m)\pi_{D_{m^2}}(x),
$$
and
$$
\mathcal{M}(x,y,2\sqrt{x}) =O\Big( \sum_{y< m \le 2\sqrt{x}}\#\{ p\le x \mid   p\nmid mN_E, a_p\neq 0,  m^2\mid |\bar{E}(\Bbb{F}_p)| \} \Big).
$$
In addition, we can consider a finer splitting
\begin{align}\label{hE-expression-finer}
 \begin{split}
h_E(x,\Bbb{Q})
 =\mathcal{N}(x,y) +\mathcal{N}_0(x, y , x^{\theta})
+\mathcal{M}(x, x^{\theta} ,2\sqrt{x})+ O\Big( \frac{x^{\frac{1}{2}+\theta}}{\log x} +x^{\theta}\log( N_E x) \Big),
 \end{split}
\end{align} 
for $\theta\in (0,\frac{1}{4})$,  where
\begin{equation}\label{def-M0}
\mathcal{N}_0(x, y , x^{\theta}) 
 =O\Big( \sum_{ y< m\le x^{\theta}} \mu(m)^2 \pi_{D_{m^2}}(x) \Big).
\end{equation}
On the other hand, reasoning similarly, we also have
$$
h_E(x,h,\Bbb{Q})=\mathcal{N}(x,h,y)+\mathcal{M}(x,h,y,2\sqrt{x}) + O\Big( \frac{\sqrt{x} y}{\log x} +y\log( yN_E) \Big),
$$
where  $h_E(x,h,\Bbb{Q})$ is defined in \eqref{def-h-SI},
$$
\mathcal{N}(x,h,y) =\frac{1}{2}\sum_{m\le y}\mu(m)\#\{\p\subset \mathcal{O}_K \mid x< \norm(\p)\le x+h, \p\nmid mN_E, \sigma_{\p}\subseteq  D_{m^2} \},
$$
and
\begin{align*}
 \begin{split}
&\mathcal{M}(x,h,y,2\sqrt{x+h})\\
& =O\Big(\sum_{y<m\le 2\sqrt{x+h}}\#\{x< p\le x+h \mid   p\nmid mN_E, a_p\neq 0,  m^2\mid |\bar{E}(\Bbb{F}_p)| \}\Big).
 \end{split}
\end{align*}

To establish asymptotic formulae for $h_E(x,\Bbb{Q})$ and $h_E(x,h,\Bbb{Q})$, it is crucial to analyse $\mathcal{N}(x,y)$ and $\mathcal{N}(x,h,y)$, respectively. In \cite{Co08}, Cojocaru applied effective versions of the Chebotarev density theorem established in \cite{LO77,MMS88} to handle $h_E(x,\Bbb{Q})$. Under GRH, the effective version due to M.R. Murty-V.K. Murty-Saradha \cite{MMS88}  provides a satisfactory error term. In contrast, the unconditional result obtained by Lagarias and Odlyzko gives a much weaker estimate, which led Cojocaru to use a more delicate sieve argument. In light of the works of  Akbary-V.K. Murty \cite {AM} and M.R. Murty \cite{MR83}, we observe that the Chebotarev conditions involved (i.e., $\sigma_{\p}\subseteq D_{m^2}$) can be further translated into congruence conditions over certain ray class groups via Artin reciprocity (see Section \ref{proof-main-thm1} for more details). This observation hints  at using the following  Bombieri-Vinogradov theorem for number fields and its short interval variant. As shall be seen, such a Bombieri-Vinogradov theorem for number fields will play a crucial role in improving Cojocaru's estimate \eqref{Co-02}.

Let $F$ be a number field, and let $\a$ and $\q$ be integral ideals of $F$.  For any $(\a,\q)=1$, we set
$$
\pi(x,\q,\a)=\#\{\p \subset \mathcal{O}_F \mid \text{$\norm(\p)\le x$ and $\p\sim\a\mymod{\q}$}\}, 
$$
where $\p\sim\a\enspace\mymod{\q}$ means that $\p$ and $\a$ belong to the same ray class of the ray class group modulo $\q$. (For the background on ray class groups, we refer the reader to \cite[Ch. 3]{Ch09}.) In  \cite{Hu86}, Huxley proved the following Bombieri-Vinogradov theorem for number fields, which improves the previous result of Wilson \cite{Wil}.
\begin{theorem}[Huxley]\label{Huxley}
Let $F$ be a number field. Then for any $A>0$, there is $B=B(A)>0$ such that for $Q\le x^{1/2}/(\log x)^{B}$, one has
\begin{align*}
\sum_{\norm(\q)\le Q}\frac{h(\q)}{\phi(\q)}\max_{(\a,\q)=1}\max_{y \le x}\Big|\pi(y,\q,\a)-\frac{1}{h(\q)}\Li (y)\Big|\ll_{F,A} \frac{x}{(\log x)^{A}},
\end{align*}
where $h(\q)$ is the cardinality of the ray class group modulo $\q$, and  $\phi(\q)$ is the number field analogue of Euler's totient function for $F$.
\end{theorem}

Also, we recall the following short interval variant established by the present author in \cite{Wong20}.

\begin{theorem}\label{BVSI-Li}
Let $F$ be a number field of degree $n_F$. For $0\le \delta<\frac{2}{5n_F}$, fix $0\le\theta<\frac{1}{5n_F+10}(2-5n_F\delta)$. Then for any $x^{1-\delta}\le h\le x$ and $A>0$, we have
\begin{align*}
\sum_{\norm(\q)\le x^{\theta}}\frac{h(\q)}{\phi(\q)}\max_{(\a,\q)=1}\Big|\pi(x+h,\q,\a)-\pi(x,\q,\a)-\frac{\Li (x+h) -\Li(x)}{h(\q)}\Big| \ll \frac{h}{(\log x)^{A}},
\end{align*}
where the implied constant depends on $F$ and $A$. 
\end{theorem}

In our consideration, we will take $F=K$, an imaginary quadratic field of class number one associated to a CM elliptic curve $E$ as discussed in the introduction. For such an instance, the factor $\frac{h(\q)}{\phi(\q)}$ can be dropped from the estimates in Theorems \ref{Huxley} and \ref{BVSI-Li} as follows. Let $h_F$ denote the class number of $F$ and  $r_1$ be the number of real embeddings of $F$. It is known that 
$$
h(\q)=\frac{h_F2^{r_1}\phi(\q)}{T(\q)},
$$
where $T(\q)$ is the number of residue classes $\mymod{\q}$ that contain a unit (see \cite[Eq. (7)]{Wil}). Now, for imaginary quadratic fields of class number one, it is known that $T(\q)\le 6$, and hence  $\frac{h(\q)}{\phi(\q)}\ge \frac{1}{6}$. Thus, the factor $\frac{h(\q)}{\phi(\q)}$ can be removed.

Secondly, we shall control $\mathcal{M}(x,y,2\sqrt{x})$ and $\mathcal{M}(x,h,y,2\sqrt{x+h})$ by the sieve methods used by Cojocaru \cite{Co08}.  As remarked by  Cojocaru, the difficulties of the proofs lie in analysing $\mathcal{M}(x,y,2\sqrt{x})$ and $\mathcal{M}(x,h,y,2\sqrt{x+h})$ since the classical  Brun-Titchmarsh inequality is not strong enough. Consequently, similar to the argument in \cite[Sec. 5]{Co08}, we will require a Brun-Titchmarsh type estimate for short intervals in Section \ref{proof-main-thm2}. 

Last but not least, we remark that instead of directly following Cojocaru's argument to control $\mathcal{N}_0(x, y , x^{\theta})$, one can derive a refined estimate for this range 
by applying  the following number field analogue of the  Brun-Titchmarsh inequality established by  Hinz and Lodemann \cite[Theorem 4]{HL93}.

\begin{theorem}[Hinz-Lodemann]\label{HLBT} In the above notation, for any $(\a,\q)=1$, if $\norm(\q) < x$, then
$$
\pi(x,\q,\a)\le \frac{2x}{h(\q)   \log ( x/ \norm(\q) )}\cdot \Big( 1 +O\Big( \frac{ \log \log ( 3x/ \norm(\q)) }{ \log ( x/ \norm(\q) )}  \Big) \Big) .
$$
\end{theorem}

\section{Proof of Theorem \ref{main-thm1}}\label{proof-main-thm1}

\subsection{Proof of the first part of Theorem \ref{main-thm1}}

Recall that by the theory of complex multiplication, M.R. Murty \cite[Lemma 4]{MR83} showed that there exists an integral ideal $\mathfrak{f}=\mathfrak{f}_E$ of $\mathcal{O}_K$ such that
$$
K_{\m} \subseteq K(E[m]) \subseteq  K_{\mathfrak{fm}},
$$
where  $K_{\m}$ (resp.,  $K_{\mathfrak{fm}}$) is the ray class field of $K$ of level $\m=m\mathcal{O}_K$ (resp., $\mathfrak{fm}$).  Moreover, as remarked in \cite{AM} (see also  \cite[p, 163]{MR83}), $\mathfrak{f}$ above can be taken as the conductor of the Hecke character associated with $E$, and thus $N_E = \norm(\mathfrak{f})d_K$, where $d_K$ is the absolute discriminant of $K$. In particular, we have $\phi(\mathfrak{f})\le\norm(\mathfrak{f})\le N_E$.

Now, we let $\tilde{D}_{m^2}$ stand for the preimage of $D_{m^2}$ under the quotient map from $\Gal( K_{\mathfrak{fm^2}}/K)$ to $\Gal( K(E[m^2])/K)$. As $\tilde{D}_{m^2}$ is a conjugacy set in $\Gal( K_{\mathfrak{fm^2}}/K)$, we can let $\pi_{\tilde{D}_{m^2}}(x)$ denote the number of unramified primes $\p$ in $ K_{\mathfrak{fm^2}}/K$ such that $\norm(\p)\le x$ and the Artin symbol $\tilde{\sigma}_{\p}$ at $\p$ is contained in $\tilde{D}_{m^2}$. It is known that
\begin{equation}\label{DD}
\pi_{D_{m^2}}(x) =\pi_{\tilde{D}_{m^2}}(x) + O\Big(\frac{\log d_{K_{\mathfrak{fm^2}}}}{[K_{\mathfrak{fm^2}}:K]}\Big),
\end{equation}
where $\pi_{D_{m^2}}(x)$ is defined as in \eqref{def-pi-D} (see, e.g., \cite[p. 268]{MMS88}). Recall that by Hensel's bound (see, e.g., \cite[Lemma 3.4]{Co08}), for any Galois extension $L/K$ of number fields, one has
$$
\log d_L \le [L:K] \log d_K +n_L \sum_p \log p  +n_L \log [L:K],  
$$
where for a number field $F$, $d_F$ and $n_F$ denote its absolute discriminant and degree, respectively, and the sum is over the primes $p$ lying below primes of $K$ that ramify in $L/K$. Thus, applying the above inequality with $L=K_{\mathfrak{fm^2}}$, we have
$$
\frac{\log d_{K_{\mathfrak{fm^2}}}}{[K_{\mathfrak{fm^2}}:K]}
\le \log d_K + n_K\log \norm(\mathfrak{fm^2}) + n_K\log [K_{\mathfrak{fm^2}}:K]
 \le \log d_K + 4\log \norm(\mathfrak{fm^2}) 
$$
as $n_K=2$ and $
[K_{\mathfrak{fm^2}}:  K]\le h(\mathfrak{fm^2}) \le 1\cdot 2^0\cdot \phi(\mathfrak{fm^2})\le \norm(\mathfrak{fm^2})$.
Hence, we arrive at 
\begin{equation}\label{DD-error-bd}
\frac{\log d_{K_{\mathfrak{fm^2}}}}{[K_{\mathfrak{fm^2}}:K]}\ll_K \log \norm(\mathfrak{fm^2}).
\end{equation}
Moreover, by Artin reciprocity (see, e.g., \cite[Ch. 3 and 5]{Ch09}), there are $u(m^2)$ ray classes  $\m_i$ modulo $\mathfrak{fm^2}$ so that
\begin{equation}\label{decomp-pi_D}
\pi_{\tilde{D}_{m^2}}(x)=\sum_{i=1}^{u(m^2)}\pi(x,\mathfrak{fm^2},\m_i ),
\end{equation}
where $u(m^2)\le h(\mathfrak{fm^2}) \le \norm(\mathfrak{fm^2})$.
It then follows from the Chebotarev density theorem and the prime ideal theorem for ray classes that
\begin{equation}\label{quotient-eq}
\frac{|D_{m^2}|}{[K(E[m^2]):K]}=\frac{|\tilde{D}_{m^2}|}{[K_{\mathfrak{fm^2}}:K]}=\frac{u(m^2)}{h(\mathfrak{fm^2})}.
\end{equation}

We are now in a position to estimate $\mathcal{N}(x,y)$. Choosing $y= z/\norm(\mathfrak{f})^{1/4}$, we write
\begin{align}\label{middle3}
 \begin{split}
\mathcal{N}(x,y)
&=\frac{1}{2}\sum_{m\le z/\norm(\mathfrak{f})^{1/4}}\mu(m)\pi_{D_{m^2}}(x) \\
&=\frac{1}{2}\sum_{ m\le z/\norm(\mathfrak{f})^{1/4}}\mu(m)\frac{|D_{m^2}|\Li(x)}{[K(E[m^2]):K]}
+\frac{1}{2}\sum_{ m\le z/\norm(\mathfrak{f})^{1/4}}\mu(m)\mathcal{ E}(x,m^2), 
  \end{split}
\end{align}
where 
$$
\mathcal{ E}(x,m^2)=\pi_{D_{m^2}}(x)-\frac{|D_{m^2}|\Li(x)}{[K(E[m^2]):K]}.
$$
By \eqref{DD}, \eqref{DD-error-bd}, \eqref{decomp-pi_D}, and \eqref{quotient-eq}, we have
\begin{align*}
 \begin{split}
\mathcal{ E}(x,m^2)
&= \pi_{\tilde{D}_{m^2}}(x)-\frac{|\tilde{D}_{m^2}|\Li(x)}{[K_{\mathfrak{fm}^2}:K]} +O(\log \norm(\mathfrak{fm^2}) )\\
&=\sum_{i=1}^{u(m^2)}\Big(\pi(x,\mathfrak{fm^2},\m_i)-\frac{\Li(x)}{h(\mathfrak{fm^2})}\Big)+O(\log \norm(\mathfrak{fm^2}) ).
  \end{split}
\end{align*}
Therefore, by Theorem \ref{Huxley}, together with the discussion beneath Theorem \ref{BVSI-Li}, for $z^4\le x^{1/2-\epsilon}$ and $A>0$,  we have
\begin{align}\label{bound-E}
  \begin{split}
&\sum_{m\le z/\norm(\mathfrak{f})^{1/4}}\big|\mathcal{ E}(x,m^2)\big|\\
&\ll \sum_{\norm(\mathfrak{fm^2})\le z^4}\Big( \sum_{i=1}^{u(m^2)}\Big|\pi(x,\mathfrak{fm^2},\m_i)-\frac{\Li(x)}{h(\mathfrak{fm^2})}\Big| + \log \norm(\mathfrak{fm^2})  \Big)\\
&\ll \sum_{\norm(\mathfrak{fm^2})\le z^4} u(m^2)\max_{1\le i\le u(m^2)} \Big|\pi(x,\mathfrak{fm^2},\m_i)-\frac{\Li(x)}{h(\mathfrak{fm^2})}\Big| + z^4 \log x   \\
&\ll  z^4\sum_{\norm(\mathfrak{q})\le z^4}\max_{(\a,\q)=1}\Big|\pi(x,\q,\a)-\frac{\Li(x)}{h(\q)}\Big| + z^4\log x\\
&\ll_{A} \frac{xz^4}{(\log x)^{5A+12}}
 \end{split}
\end{align}
as $u(m^2)\le \norm(\mathfrak{fm^2})$. Since there are only nine imaginary quadratic fields $K$ of class number one, we can make the implied constant independent of $K$.

By \cite[Eq. (36)]{Co08}, one has
\begin{equation}\label{(36)-Co08}
\mathcal{M}(x,y,2\sqrt{x})\ll \frac{x}{y}(\log x)^3 +\sqrt{x} (\log x)^4.
\end{equation}
Also, by \cite[Eq. (39)]{Co08}, one knows
$$
\sum_{m>y} \frac{\mu(m)|D_{m^2}|}{[K(E[m^2]):K]}\Li(x) 
\ll
 \frac{x \log y}{y \log x}.
$$
As $y= z/\norm(\mathfrak{f})^{1/4}$ and $\norm(\mathfrak{f}) \le N_E$, the above two estimates, together with \eqref{hE-expression}, \eqref{middle3}, and \eqref{bound-E}, then give
\begin{align*}
h_E(x,\Bbb{Q})
&= \frac{1}{2}\sum_{m=1}^{\infty} \frac{\mu(m)|D_{m^2}|}{[K(E[m^2]):K]}\Li(x) + O_{A}\Big( N_E\frac{xy^4}{(\log x)^{5A+12}}\Big)\\
&+O\Big(\frac{x}{y}(\log x)^3 +\sqrt{x} (\log x)^4 + \frac{x \log y}{y \log x} +\frac{\sqrt{x} y}{\log x} +y\log( yN_E) \Big).
\end{align*}
Finally, choosing  $y=(\log x)^{A+3}$, we conclude that for any $A>0$,
$$
h_E(x,\Bbb{Q})= \frac{1}{2}\sum_{m=1}^{\infty} \frac{\mu(m)|D_{m^2}|}{[K(E[m^2]):K]}\Li(x) + O_{A}\Big( N_E \frac{x}{(\log x)^A}\Big),
$$
which, together with the identity 
\begin{equation}\label{deltaE-id}
\frac{1}{2} \sum_{m=1}^{\infty} \frac{\mu(m)|D_{m^2}|}{[K(E[m^2]):K]} =  \frac{1}{2} \sum_{\a,m}\frac{\mu(\a)\mu(m)}{[K(E[\a^2])K(E[m]):K]}
\end{equation}
(see \cite[Eq. (11) and (40)]{Co08}),
yields the claimed asymptotic formula \eqref{main-1}.

\subsection{Proof of the second part of Theorem \ref{main-thm1}}

The proof of \eqref{improved-hQ} follows closely  Cojocaru's argument. The main difference lies in bounding $\mathcal{N}_0(x, y , x^{\theta}) $ by Theorem \ref{HLBT}, the  Brun-Titchmarsh inequality for number fields due to Hinz and Lodemann.

We begin by recalling that as $K(E[m^2])/ K$ is an abelian Galois extension, Artin's (holomorphy) conjecture is proven for this case. Applying the effective Chebotarev density theorem established by  M.R. Murty, V.K. Murty, and Saradha \cite{MMS88}, Cojocaru  \cite[Eq. (34)]{Co08} derived that
\begin{equation*}
\mathcal{N}(x,y)
=
\frac{1}{2}\sum_{m\le y} \frac{\mu(m)|D_{m^2}|}{[K(E[m^2]): K]} \Li(x)
+ O( y^2 x^{1/2} \log(N_E x )),
\end{equation*}
for $y\le 2\sqrt{x}$,  under GRH. Hence, by \eqref{deltaE-id}, one knows
\begin{equation}\label{new-est-N}
\mathcal{N}(x,y)
=  \delta_E \Li(x) + O( y^2 x^{1/2} \log(N_E x ))-  \frac{1}{2}\sum_{m>y  } \frac{\mu(m)|D_{m^2}|}{[K(E[m^2]): K]} \Li(x) .
\end{equation}

In \cite[Lemma 3.12]{Co08}, Cojocaru showed that $|D_{q^2}|\le q^2$ for any odd prime $q$ that is unramified in $K$. From this, she deduced in \cite[Corollary 3.13]{Co08}  that $|D_{m^2}|\le m^2$ for any odd positive square-free integer $m$ composed of primes which are unramified in $K$. We remark that this bound can be extended to all square-free integers $m>1$ so that $|D_{m^2}|\ll m^2$ as follows.\footnote{In fact, such an extension has been used implicitly in \cite[Eq. (34) and (39)]{Co08}.} Observe that for any prime $\ell$ that is even or ramifies in $K$, we always have
$$
|D_{\ell^2}|\le |\Gal(K(E[\ell^2])/K )| \ll \ell^4.
$$
As only finitely many primes ramify in $K$, we can write this bound as $|D_{\ell^2}|\ll_K  1$. In addition, since $K$ must be one of nine imaginary quadratic fields of class number one, we deduce that $|D_{\ell^2}|$ is absolutely bounded (i.e.  $|D_{\ell^2}|\ll  1$) for any prime $\ell$ that is even or ramifies in $K$. This, combined with the above-mentioned result of Cojocaru, yields $|D_{m^2}|\ll m^2$ for any square-free integer $m>1$. As a direct consequence, we have
\begin{equation}\label{new-tail-bd}
\sum_{m> y} \frac{\mu(m)|D_{m^2}|}{[K(E[m^2]): K]}
\ll \sum_{m> y} \frac{\mu(m)^2|D_{m^2}|}{[K(E[m^2]): K]}
\ll \sum_{m> y} \frac{m^2}{\phi(m^2)^2 }
= \sum_{m> y} \frac{1}{\phi(m)^2 }
\ll \frac{1}{y},
\end{equation}
where we used the bound $[K(E[k]): K]\gg \phi(k)^2$ for any integer $k \ge 3$ (see, e.g., \cite[Eq. (19)]{Co08}). Hence, the last sum in \eqref{new-est-N} is $\ll \frac{x}{y\log x}$ (cf. \cite[Eq. (39)]{Co08}). 

Now, we shall estimate $\mathcal{N}_0(x, y , x^{\theta}) $ introduced as in \eqref{def-M0} with $y\le x^{\theta}$ and $\theta\in(0,\frac{1}{4})$. To do so,  it suffices to bound
$$
 \sum_{ y< m\le x^{\theta}} \mu(m)^2 \pi_{D_{m^2}}(x) .
$$
It follows from \eqref{DD}, \eqref{DD-error-bd}, and \eqref{decomp-pi_D}, this sum is
$$
\sum_{ y< m\le x^{\theta}}  \mu(m)^2 \sum_{i=1}^{u(m^2)}\pi(x,\mathfrak{fm^2},\m_i )  + O( x^{\theta}\log (N_E x)).
$$
Moreover, applying Theorem \ref{HLBT}, when $N_E\le  x^{1 - 4\theta -\epsilon}$ (so that $\norm( \mathfrak{fm^2} )\le N_E m^4   \le x^{1 - \epsilon} $ for $m\le x^{\theta}$), we can bound the  double sum above by
$$
\ll_\epsilon  \sum_{ y< m\le x^{\theta}} \mu(m)^2  \sum_{i=1}^{u(m^2)} \frac{x}{h( \mathfrak{fm^2} ) \log x}
=   \sum_{ y< m\le x^{\theta}}    \frac{\mu(m)^2|D_{m^2}|}{[K(E[m^2]):K]} \frac{x}{\log x},
$$
 where the  equality is due to \eqref{quotient-eq}. Therefore, by \eqref{new-tail-bd}, we deduce that
$$
\mathcal{N}_0(x, y , x^{\theta})
\ll_\epsilon \frac{x }{y\log x} +x^{\frac{1}{4}}
\ll    \frac{x }{y\log x} 
$$
upon recalling the assumption that $y< x^{\theta}$ with $\theta\in(0,\frac{1}{4})$. As \eqref{(36)-Co08} gives
$$
\mathcal{M}(x, x^{\theta} ,2\sqrt{x})
= \frac{x }{x^\theta} (\log x)^3 + \sqrt{x} (\log x)^4,
$$
we then arrive at
$$
h_E(x,\Bbb{Q}) =\delta_E \Li(x) +
O_\epsilon\Big( y^2 x^{1/2} \log(N_E x ) +  \frac{x}{y \log x} +    \frac{x}{x^\theta} (\log x)^3 + \sqrt{x} (\log x)^4 \Big).
$$
Therefore, choosing $\theta =\frac{1}{6} +\epsilon$ and
$$
y =\frac{x^{1/6}}{ (\log x )^{1/3}(\log (N_E x) )^{1/3}},
$$
we obtain 
$$
h_E(x,\Bbb{Q}) =\delta_E \Li(x) +
O_\epsilon\Big( x^{5/6}  \frac{(\log(N_E x ) )^{1/3} }{ (\log x)^{2/3}}\Big)
$$
whenever $N_E\le  x^{1 - \frac{4}{6} -5\epsilon}$, as desired.

\section{Proof of Theorem \ref{main-thm2}}\label{proof-main-thm2}

We start by setting $y= z/\norm(\mathfrak{f})^{1/4}$ and writing
\begin{align*} 
 \begin{split}
&\mathcal{N}(x,h,y)\\ 
&= \frac{1}{2}\sum_{m\le z/\norm(\mathfrak{f})^{1/4}}\mu(m)(\pi_{D_{m^2}}(x+h)-\pi_{D_{m^2}}(x)) \\
&=\frac{1}{2}\sum_{m\le z/\norm(\mathfrak{f})^{1/4}}\mu(m)\frac{|D_{m^2}|(\Li(x+h)-\Li(x))}{[K(E[m^2]):K]}+\frac{1}{2}\sum_{ m\le z/\norm(\mathfrak{f})^{1/4}}\mu(m)\tilde{\mathcal{ E}}(x,m^2), 
  \end{split}
\end{align*}
where 
$$
\tilde{\mathcal{ E}}(x,m^2)=\pi_{D_{m^2}}(x+h)-\pi_{D_{m^2}}(x)-\frac{|D_{m^2}|(\Li(x+h)-\Li(x))}{[K(E[m^2]):K]}.
$$
As argued in Section \ref{proof-main-thm1}, $\tilde{\mathcal{ E}}(x,m^2)$ is equal to
\begin{align*}
 \begin{split}
& \pi_{\tilde{D}_{m^2}}(x+h)-\pi_{\tilde{D}_{m^2}}(x)-\frac{|\tilde{D}_{m^2}|(\Li(x+h)-\Li(x))}{[K_{\mathfrak{fm}^2}:K]} +O(\log \norm(\mathfrak{fm^2}) )\\
&=\sum_{i=1}^{u(m^2)}\Big(\pi(x+h,\mathfrak{fm^2},\m_i)-\pi(x,\mathfrak{fm^2},\m_i)-\frac{\Li(x+h)-\Li(x)}{h(\mathfrak{fm^2})}\Big)+O(\log \norm(\mathfrak{fm^2}) ).
  \end{split}
\end{align*}
Now, for $\delta\in [0,\frac{1}{5})$, $\theta\in(0,\frac{1-5\delta}{10})$, and $z^4\le x^{\theta}$, applying Theorem \ref{BVSI-Li}, for any $A>0$, we derive
\begin{align*}
  \begin{split}
& \sum_{ \norm(\mathfrak{fm^2})\le z^4}\sum_{i=1}^{u(m^2)}\Big|\pi(x+h,\mathfrak{fm^2},\m_i)-\pi(x,\mathfrak{fm^2},\m_i)-\frac{\Li(x+h)-\Li(x)}{h(\mathfrak{fm^2})}\Big| \\
&\ll   z^4\sum_{\norm(\mathfrak{q})\le z^4}\max_{(\a,\q)=1}\Big|\pi(x+h,\q,\a)-\pi(x,\q,\a)-\frac{\Li(x+h)-\Li(x)}{h(\q)}\Big| \\
&\ll_{A} N_E\frac{h y^4}{(\log x)^{5A+13}},
 \end{split}
\end{align*}
whenever $x^{1-\delta}\le h\le x$. Thus, we arrive at
$$
\sum_{ m\le z/\norm(\mathfrak{f})^{1/4}}\big|\tilde{\mathcal{ E}}(x,m^2)\big|\ll_{A}  N_E \frac{hy^4}{(\log x)^{5A+12}}.
$$

As remarked previously, in order to adapt the sieve method developed by  Cojocaru in \cite[Sec. 5]{Co08} (to control $\mathcal{M}(x,y,2\sqrt{x})$), we shall require a number field analogue of the Brun-Titchmarsh inequality as follows. Assume $K=\Bbb{Q}(\sqrt{-D})$. For an ordinary prime $p\nmid 6N_E$, we write  $p\mathcal{O}_K =(\pi_p)(\bar{\pi}_p)$ and  note that $\norm(\pi_p)=p$.  Observe that
\begin{equation}\label{trival-short-ineq}
(\sqrt{x}-1)^2 < (\sqrt{p}-1)^2 \le \norm(\pi_p-1)\le (\sqrt{p}+1)^2  \le (\sqrt{x+h} +1)^2
\end{equation}
for $x< p \le x+h$. Following  Cojocaru \cite{Co08}, we shall use the decomposition $m=m_i m_r m_s$,
where $m_i$ is composed of primes inert in $K$,  $m_r$ is composed of primes ramifying in $K$, and $m_s$ is composed of primes splitting completely in $K$. It follows from  \cite[Lemma 3.17]{Co08} that
for any square-free $m\in \Bbb{N}$ such that $m^2$ divides $|\bar{E}(\Bbb{F}_p)|$ for some ordinary prime $p$, one knows 
\begin{equation}\label{alpha}
(\pi_p-1) = m_i m_r I( m_s)\cdot (\alpha) 
\end{equation} 
for some $\alpha\in \mathcal{O}_K$. Here, $I(k_s)$ is an ideal in $\mathcal{O}_K$ obtained from the product of the ideals $\q\bar{\q}$, $\q^2$, or $\bar{\q}^2$ in accordance to whether $\pi_p$ splits completely in $K(E[\q\bar{\q}])$, $K(E[\q^2])$, or $K(E[\bar{\q}^2])$, respectively, while $\q$ runs over the primes  of $K$ above prime divisors $q$ of $m_s$. (As remarked in \cite[p. 610]{Co08},  there are $3^{\nu(m_s)}$ possibilities of such ideals $I(k_s)$.) Hence, for $x< p\le x+h$, by \eqref{trival-short-ineq} and \eqref{alpha}, we have
$$
 \frac{(\sqrt{x} -1)^2}{m^2} \le\norm(\alpha)= \frac{\norm(\pi_p-1)}{\norm(m_i m_r I( m_s))}\le  \frac{(\sqrt{x+h} +1)^2}{m^2}.
$$
The number of such $\alpha$ is 
$$
\ll\Big(\frac{\sqrt{x+h}- \sqrt{x}+2}{m}+1 \Big)\Big(\frac{\sqrt{x+h}- \sqrt{x}+2}{m\sqrt{D}}+1  \Big)\ll \Big(\frac{\sqrt{x+h}- \sqrt{x}}{m}+1 \Big)^2,
$$
which is
$$
\ll \frac{h}{m^2}+\frac{ \sqrt{x+h}}{m}  
$$
for $ m\le 2\sqrt{x+h}$.
(Here, we used $(\sqrt{x+h}- \sqrt{x})^2= h +2 (x-\sqrt{x+h}\sqrt{x} )\ll h$. Note that if $x/h$ tends to $\infty$ as $x\rightarrow \infty$, we have
$$
\lim_{x\rightarrow\infty}\frac{x-\sqrt{x+h}\sqrt{x}}{h}= \lim_{v\rightarrow\infty} v\Big( 1- \sqrt{1+\frac{1}{v}}\Big) = \lim_{w\rightarrow 0}\frac{1- \sqrt{1+w}}{w}=-\frac{1}{2},
$$
where the last equality is due to  l'H\^{o}pital's rule.)

Processing an argument similar to \cite[Eq. (36)]{Co08}, we have
\begin{align}\label{M-h-mid}
  \begin{split}
\mathcal{M}(x,h,y,2\sqrt{x+h})
&\le \sum_{\substack{y<m\le 2\sqrt{x+h}\\ \text{$m$ square-free}\\ m=m_i m_r m_s}}\sum_{\substack{x< p\le x+h\\ a_p\neq 0\\ (\pi_p -1 )=m_i m_r I( m_s)\cdot(\alpha)}}1\\ 
&\ll \sum_{\substack{y<m\le 2\sqrt{x+h}\\ \text{$m$ square-free}\\ m=m_i m_r m_s}}3^{\nu(m_s)}\Big( \frac{h}{m_i^2 m_r^2 m_s^2} +\frac{\sqrt{x}}{m_i m_r m_s}\Big) \\
&\ll \sum_{m_s\le 2\sqrt{x+h}}\sum_{\frac{y}{m_s}<\tilde{m} \le 2\sqrt{x+h}}3^{\nu(m_s)}\Big( \frac{h}{\tilde{m}^2 m_s^2} +\frac{\sqrt{x}}{\tilde{m}  m_s}\Big)\\
&\ll \sum_{m_s\le 2\sqrt{x+h}} 3^{\nu(m_s)}\Big( \frac{h}{y m_s} +\frac{\sqrt{x}}{  m_s} \log x\Big)\\
&\ll \frac{h}{y}(\log x)^3 +\sqrt{x} (\log x)^4,
 \end{split}
\end{align}
where the last bound follows from the elementary estimates $\sum_{m\le x} \frac{1}{m} =\log x +O(1)$ and
$$
\sum_{m\le x} 3^{\nu(m)}\ll x(\log x)^{2}.
$$
Also, as in \cite[Eq. (39)]{Co08}, it can be checked that
$$
\sum_{m>y} \frac{\mu(m)|D_{m^2}|}{[K(E[m^2]):K]}(\Li(x+h)-\Li(x)) 
\ll \frac{h \log y}{y \log x}.
$$

Now, gathering everything together, we deduce 
\begin{align*}
h_E(x,h,\Bbb{Q})&= \frac{1}{2} \sum_{m=1}^{\infty} \frac{\mu(m)|D_{m^2}|}{[K(E[m^2]):K]}(\Li(x+h)-\Li(x) ) + O_{A}\Big( N_E\frac{hy^4}{(\log x)^{5A+12}}\Big)\\
&+O\Big(\frac{h}{y}(\log x)^3 +\sqrt{x} (\log x)^4 + \frac{h \log y}{y \log x} +\frac{\sqrt{x} y}{\log x} +y\log( yN_E) \Big).
\end{align*}
Finally, choosing  $y=(\log x)^{A+3}$, we conclude that for any $A>0$,
$$
h_E(x,h,\Bbb{Q})= \frac{1}{2} \sum_{m=1}^{\infty} \frac{\mu(m)|D_{m^2}|}{[K(E[m^2]):K]}(\Li(x+h)-\Li(x) )+ O_{A}\Big( N_E \frac{h}{(\log x)^A}\Big)
$$
whenever $x^{1-\delta}\le h\le x$ with $\delta<\frac{1}{5}$. Recalling the identity 
\eqref{deltaE-id}, we conclude the proof.

\section{Proof of Theorem \ref{main-thm-cyclicity}}\label{proof-main-thm3}

Following \cite[Sec. 3]{MR83}, by an application of the inclusion-exclusion principle, one has
$$
\pi_{c}(x,E)=\#\{p \le x \mid \text{$p\nmid N_E$ and $\bar{E}(\Bbb{F}_p)$ is cyclic}\}=\sum_{m=1}^{2\sqrt{x}}\mu(m)\pi_{E}(x,m),
$$
where $\pi_{E}(x,m)$ denotes the number of primes $p \le x$ such that $p\nmid N_E$ splits completely in $\Bbb{Q}(E[m])$.

In \cite[Sec. 4]{Wong20}, the present author used the estimate 
\begin{equation}\label{old-up-bd}
\pi_{E}(x,m)\ll \frac{x}{m^2}
\end{equation}
(for square-free $3\le m\le 2\sqrt{x}$)
to obtain
\begin{align}\label{old-est-cyclicity}
 \sum_{z/\norm(\mathfrak{f})^{1/2}\le m\le 2\sqrt{x+h}} \pi_{E}(x+h,m)-\pi_{E}(x,m)
\ll \sum_{z/\norm(\mathfrak{f})^{1/2}\le m\le 2\sqrt{x+h}}\frac{x}{m^2}
\ll N_E^{1/2} \frac{x}{z}.
\end{align}
The key to proving  Theorem \ref{main-thm-cyclicity} is to improve \eqref{old-up-bd} (and thus \eqref{old-est-cyclicity}) as follows. Observe
\begin{align}\label{mid-step-1}
 \begin{split}
&\pi_{E}(x+h,m)-\pi_{E}(x,m)\\
&\le \#\{\pi_p\in\mathcal{O}_K \mid p\nmid N_E,x< \norm(\pi_p)=p\le x+h, \pi_p\equiv 1\mymod{m\mathcal{O}_K} \} +2.
 \end{split}
\end{align}
(Here, we use the fact that for square-free $m\ge 3$, if $p\nmid 6N_E$ splits completely in $\Bbb{Q}(E[m])$, then $p\mathcal{O}_K =(\pi_p)(\bar{\pi}_p)$ and $\pi_p\equiv 1\mymod{m\mathcal{O}_K}$. See \cite[Lemmata 2.4 and 2.5]{AM}.) 
Now, for $3\le m\le \sqrt{x+h}$, if $x< \norm(\pi_p)=p \le x+h $ and
$$
(\pi_p-1) =  m \cdot(\alpha)
$$
for some $\alpha \in \mathcal{O}_K$, then it follows from \eqref{trival-short-ineq} that
$$
 \frac{(\sqrt{x} -1)^2}{m^2} \le\norm(\alpha)= \frac{\norm(\pi_p-1)}{\norm(m)}\le  \frac{(\sqrt{x+h} +1)^2}{m^2}.
$$
As shown in the previous section,  the number of such $\alpha$ is $\ll \frac{h}{m^2}+\frac{ \sqrt{x+h}}{m}$. Hence, by \eqref{mid-step-1},  for $3\le m\le 2\sqrt{x+h}$, we deduce
\begin{equation}\label{NTA-BT}
\pi_{E}(x+h,m)-\pi_{E}(x,m) \ll \frac{h}{m^2}+\frac{ \sqrt{x+h}}{m}.
\end{equation} 
This leads us to
\begin{align*}
\pi_{c}(x+h,E)-\pi_{c}(x,E)
&= \sum_{1\le m\le z/\norm(\mathfrak{f})^{1/2}}\mu(m) (\pi_{E}(x+h ,m) - \pi_{E}(x ,m) )\\
 &+  O( N_E^{1/2} hz^{-1} + \sqrt{x} \log x  ).
\end{align*}

Now, controlling the range $1\le m\le z/\norm(\mathfrak{f})^{1/2}$ as in \cite[Eq. (4.2)-(4.4)]{Wong20}, we see that $\pi_{c}(x+h,E)-\pi_{c}(x,E)$ equals
\begin{align*}
&\sum_{1\le m\le z/\norm(\mathfrak{f})^{1/2}}\frac{\mu(m)}{n(m)}(\Li(x+h)-\Li(x))\\
&+O\Big(N_E^{1/2} hz^{-1}+ \sqrt{x} \log x +N_E(\log N_E)\frac{h}{(\log x)^A}+\frac{\sqrt{x}z}{\log x} +(\log N_E)z\Big),
\end{align*}
where $n(m)$ denotes the degree of $\Bbb{Q}(E[m])$, and the implied constant depends on $\Bbb{Q}(E[2])$ and $A$, whenever $z^2\le x^\theta$  and $x^{1-\delta}\le h\le x$  with $0\le\theta<\frac{1}{10}(1-5\delta)$ and $0\le \delta< \frac{1}{5}$.

Finally, from the last paragraph of \cite[Sec. 4]{Wong20}, we know
$$
\sum_{ m> z/\norm(\mathfrak{f})^{1/2}}\frac{\mu(m)}{n(m)} (\Li(x+h)-\Li(x))\ll  \frac{N_E^{1/2} h }{z}.
$$
Hence, for any $0\le \delta<\frac{1}{5}$ and $x^{1-\delta}\le h\le x$, balancing the errors, we obtain
\begin{align*}
\pi_{c}(x+h,E)-\pi_{c}(x,E)
=\c_E(\Li(x+h)-\Li(x))+O\Big(N_E(\log N_E)\frac{h}{(\log x)^A}\Big),
\end{align*}
which completes the proof.

\section*{Acknowledgments}
The author would like to thank Professors Wen-Ching Winnie Li and Robert C. Vaughan for their suggestions. He is also very grateful to the referee for the careful reading and constructive comments.

\end{document}